\documentclass[10pt]{amsart}
\usepackage{amsmath,amsfonts,amsthm,amssymb,amscd,bbm}
\usepackage{enumitem}

\newcommand{\beq}{\begin{equation}}
\newcommand{\ee}{\end{equation}}

\theoremstyle{plain}
\newtheorem{theorem}{Theorem}[section]
\newtheorem{theomain}{Theorem}
\newtheorem{lemma}[theorem]{Lemma}
\newtheorem{proposition}[theorem]{Proposition}
 \theoremstyle{definition}
 \theoremstyle{remark}
\newtheorem{remark}[theorem]{Remark}

\newcommand{\R}{{\mathbb R}}

\newcommand{\Z}{{\mathbb Z}}

\def\im{{\rm i}}

\newcommand{\C}{\mathbb{C}}
\newcommand{\T}{\mathbb{T}}

\def\uno{{\kern+.3em {\rm 1} \kern -.22em {\rm l}}}

\numberwithin{equation}{section}

\begin{document}
\title[On the half-wave equation]{On the regularity of the flow map associated with the 1D cubic periodic Half-Wave equation }
\author{Vladimir Georgiev}
\author{Nikolay Tzvetkov}
\author{Nicola Visciglia}

\address{ Vladimir Georgiev Department of Mathematics, University of Pisa, Largo Bruno Pontecorvo 5, 56100 Pisa, Italy}
\email{georgiev@dm.unipi.it}
\thanks{The first author was supported in part by Contract FIRB " Dinamiche Dispersive: Analisi di Fourier e Metodi Variazionali.", 2012, by INDAM, GNAMPA - Gruppo Nazionale per l'Analisi Matematica, la Probabilita e le loro Applicazion and by Institute of Mathematics and Informatics, Bulgarian Academy of Sciences.}

\address{Nikolay Tzvetkov, University of Cergy-Pontoise, UMR CNRS 8088, Cergy-Pontoise, F-95000}
\email{nikolay.tzvetkov@u-cergy.fr}
\thanks{The second author was supported in part by the ERC grant Dispeq.}

\address{Nicola Visciglia, Department of Mathematics, University of Pisa, Largo Bruno Pontecorvo 5, 56100 Pisa, Italy}
\email{viscigli@dm.unipi.it}
\thanks{The third author was supported in part by Contract FIRB " Dinamiche Dispersive: Analisi di Fourier e Metodi Variazionali.", 2012.}
\maketitle

\begin{abstract}
We prove that the solution map associated with the $1D$ half-wave cubic equation
in the periodic setting
cannot be uniformly continuous on bounded sets of the periodic Sobolev spaces $H^s$ with $s\in (1/4, 1/2)$.
\end{abstract}

\section{ Introduction}
Along this paper we shall consider the Cauchy problem associated with the defocusing cubic half-wave equation in the periodic setting:
\begin{equation}\label{eq:PotNLW}
\begin{cases}
 (\im  \partial _t  - |D _x|    ) u =
|u| ^{2} u, \quad \quad (t, x)\in \R\times \R/(2\pi \Z)\\
u(0, x)=f(x)\in H^s
\end{cases}
\end{equation}
where $H^s$ denote the usual $2\pi$-periodic
Sobolev spaces. We are interested in the regularity (local in time)
of the solution map associated with the Cauchy problem above.

We recall first the classical notion of local well-posedness in the Hadamard sense.
We say that the  problem  \eqref{eq:PotNLW} is well-posed
in $H^s$ if for every $R>0$ there is $T>0$ such that for every $ f \in H^s$, $\|f\|_{H^s} \leq R$ there is a unique (in a suitable framework)
solution $u(t,x) \in C([0,T];H^s)$ of \eqref{eq:PotNLW}
so that the solution map
\begin{equation}\label{map}
B^s(R) \ni f\rightarrow u(t,x) \in \mathcal{C}([0,T];H^s )
\end{equation}
is continuous, where $B^s(R)$ denotes the ball of radius $R$ centered at the origin of $H^s$.
This notion can be naturally extended to define well-posedness in subsets of $H^s$ or other functional settings.

Via standard energy estimates, in conjunction with the Sobolev embedding
$H^s\subset L^\infty$ for $s>1/2$, it is easy to show that the Cauchy problem
\eqref{eq:PotNLW} is locally well-posed
in $H^s$ with $s>1/2$.
Moreover the solution map \eqref{map} is Lipschitz continuous on bounded sets of $H^s$.
By an adaptation of a classical argument of Br\'ezis-Gallou\"et \cite{BG} one can show (see \cite[Proposition~1]{GG2014})
that the Cauchy problem \eqref{eq:PotNLW}
is globally well-posed in the space $H^{1/2}$ (this space is related to the energy conservation law),
and the corresponding solution map \eqref{map} is continuous in $H^{1/2}$.

The aim of this paper is to treat the case of initial data in $H^s$ with $s<1/2$. We prove an "ill-posedness" result, in the sense that
the corresponding solution map associated with \eqref{eq:PotNLW} cannot be uniformly continuous on bounded sets
with respect to the $H^s$ topology for $s\in(1/4,1/2)$.
More precisely our main result is the following.
\begin{theomain} \label{theMi}
For  any $ s \in (1/4, 1/2)$ one can find two sequences of
initial conditions $f_n, \tilde f_n\in C^\infty(\T)$
such that $f_n$ and $\tilde f_n$ are uniformly bounded in $H^s$
and
$$
\lim_{n \rightarrow \infty} \| f_n - \widetilde{f_n} \|_{H^s} = 0,
$$
but the corresponding solutions
$u_n(t, x)$ and $ \tilde u_n(t, x)$ of \eqref{eq:PotNLW} with data $f_n$ and
$\tilde f_n$ respectively satisfy
$$
\forall\, T>0,\quad
\liminf_{n \rightarrow \infty} \|u_n - \tilde u_n \|_{L^\infty([0,T];H^s)}  > 0.
$$
\end{theomain}
As a consequence of  Theorem~\ref{theMi}, we deduce that the Cauchy problem \eqref{eq:PotNLW} cannot be well-posed in $H^s$ for $s \in (1/4,1/2),$ with a solution map which is uniformly continuous on bounded sets in $H^s$ (in contrast with the case $s>1/2$).
We believe that the restriction $s>1/4$ in Theorem \ref{theMi} is technical, and could be removed.
Similar arguments can be applied for the focusing case, in this case the solutions $u_n$ and $\tilde u_n$ are only defined on short time intervals but still long enough
to observe the instability phenomenon displayed by Theorem~\ref{theMi}. 
For further results about the existence of minimal mass blow-up  solutions in the focusing case we refer to \cite{KLR}.

It is worth noticing that the cubic half-wave equation is rescaling invariant, and the corresponding critical space is $L^2$.
Hence there is a gap of $1/2$ derivative between the regularity expected by the scaling argument and the $H^s$
regularity needed to guarantee uniform continuity of the flow.

For previous results in the spirit of Theorem~\ref{theMi}, we refer to
\cite{Ba, BKPSV,BPS, BGT1, BGT2,CCT1,DG, KPV, KT,KT2, L1}. For results where one contradicts the continuity of the flow map, we refer to
\cite{AC, BGT3,CCT2003, L2,Lind,Tz04}.

We underline that our result gives only a partial progress on the question of the  well-posedness in Hadamard sense of \eqref{eq:PotNLW}
in $H^s$ with $s\in (0,1/2)$. In fact, it is unclear whether or not the Cauchy problem \eqref{eq:PotNLW} is well-posed with (only) a continuous solution map  for $ s \in (0,1/2)$.

Our approach to prove Theorem~\ref{theMi} is based on the observation that for short times one may approximate the solutions of \eqref{eq:PotNLW} with the solution of the so called  Szeg\H{o} equation introduced in \cite{GG2009}. Thanks to \cite{GG2009}, we know that the Szeg\H{o} equation has solutions displaying the phenomenon described by Theorem~\ref{theMi} and the time of the validity of the approximations between \eqref{eq:PotNLW} and these solutions of the  Szeg\H{o} equation is just enough to be able to transfer the property of the  Szeg\H{o} equation to \eqref{eq:PotNLW}.  Our argument here is close to the ill-posedness results obtained in
\cite{BGT3,Tz04}. However in \cite{BGT3,Tz04} one gets approximations at time scales imposed by the scaling of the corresponding equation. The main novelty in the analysis we present here is that we are able to go beyond the times imposed by the scaling thanks to a smoothing property of the problem
\begin{equation}\label{remm}
 (\im  \partial _t  - |D _x|    ) u =F(t),
\end{equation}
 where $F(t)$ oscillates on very particular time frequencies (see Lemma~\ref{l.Iit} below). This type of smoothing property is a general feature not restricted to the particular structure of \eqref{remm}. We hope therefore that such an argument may be useful in other contexts.

 It is observed in \cite{GG2014} that the Szeg\"o equation is the resonant part of the half-wave equation \eqref{eq:PotNLW}. This observation is of importance for long time questions concerning \eqref{eq:PotNLW}, it is however not clear to us how to exploit it in the context of high frequency short time problems as the one treated in
 Theorem~\ref{theMi}.

 The remaining part of the paper is devoted to the proof of  Theorem~\ref{theMi}.
\section{Special solutions to the Szeg\H{o} equation}\label{sez}

First of all we introduce the Cauchy problem associated with the
Szeg\H{o} equation in the periodic setting:
\begin{equation}\label{eq.I0}
\begin{cases}    \im \partial_t \mathcal{V}(t,x) =  P_{\geq 0}(\mathcal{V}(t,x) |\mathcal{V}(t,x)|^2), \quad \quad (t, x)\in \R\times \R/(2\pi \Z)\\
u(0, x)=f(x)\in H^s
\end{cases}
\end{equation}
where
$P_{\geq 0}$ is the projection operator on positive frequencies:
\begin{equation}\label{eq.IP1}
   P_{\geq 0} \Big( \sum_{k \in \Z} \widehat{f}(k) e^{\im  k x}  \Big) =\sum_{k =0}^\infty \widehat{f}(k) e^{\im k x}.
\end{equation}
Next we recall an explicit family of solutions $\mathcal{V}(t,x)$ to Szeg\H{o} equation \eqref{eq.I0}
introduced in \cite{GG2009}:
\begin{equation}\label{eq.I1}
  \mathcal{ V} (t,x) = e^{-\im t \omega} \varphi_{\alpha, p}( e^{-\im ct} e^{\im x}),
\end{equation}
where  the function $ \varphi_{\alpha, p}(z)$ is given by
\begin{equation}\label{eq.I2}
   \varphi_{\alpha, p}(z)= \frac{\alpha}{1-pz}
\end{equation}
 for $\alpha \in \R, \alpha \neq 0,$ $p \in (0,1)$ and $z \in \C, |z|=1,$
and  \begin{equation}\label{eq.I3}
  \omega = \frac{\alpha^2}{(1-p^2)^2}, \ c= \frac{\alpha^2}{1-p^2}.
 \end{equation}
By using this family of solutions it is possible to show that the solution map associated with
\eqref{eq.I0} cannot be uniformly continuous
in the space $H^s$ for $0<s<1/2$.

More precisely let $s\in (0, 1/2)$ be fixed. Then for any
given $\varepsilon \in (0,1)$ we fix the parameters
\begin{equation}\label{eq.I4}
   p = p(\varepsilon) = \sqrt{1-\varepsilon},
\end{equation}
and
\begin{equation}\label{eq.I5}
     \alpha_1 = \alpha_1(\varepsilon) = \varepsilon^{s+1/2},\  \alpha_2 =\alpha_2(\varepsilon) = \varepsilon^{s+1/2}(1+\delta(\varepsilon)),
\end{equation}
where $\delta(\varepsilon) = |\log \varepsilon|^{-1/4}.$
 We can construct now two families of solutions to the Szeg\"o equation as follows:
 \begin{equation}\label{eq.I7}
   \mathcal{ V}^{(j)}_\varepsilon(t,x) = e^{-\im t \omega_j} \varphi_{\alpha_j, p}( e^{-\im c_j t} e^{\im x}),\quad \quad j=1,2,
\end{equation}
where $\omega_j=\omega_j(\varepsilon)$ and $c_j=c_j(\varepsilon)$ for $j=1,2$
are given by \eqref{eq.I3}.
For these solutions we can apply the argument of \cite[Section~5]{GG2009} and one can show
the following estimates.
\begin{proposition} Let $0 < s < 1/2.$ 
Then there exist $\varepsilon_0 >0$ and $D>0,$ such that
\begin{equation}\label{eq.con1}
    \lim_{\varepsilon \searrow 0} \|{\mathcal V}^{(1)}_{\varepsilon}(0, \cdot) - \mathcal{V}^{(2)}_{\varepsilon}(0, \cdot)\|_{H^s} = 0
     \end{equation}
and for any $ \varepsilon \in (0, \varepsilon_0]$
\begin{equation}\label{eq.con2}
   \|\mathcal{V}^{(1)}_{\varepsilon}(t_\varepsilon, \cdot) -
   \mathcal{V}^{(2)}_{\varepsilon}(t_\varepsilon, \cdot)\|_{H^s} \geq  D >0
\end{equation}
for  $ t_\varepsilon = \varepsilon^{1-2s}|\log(\varepsilon)|^{1/2}.
$
\end{proposition}

\begin{proof}
We have the relations
$$ \left\| \mathcal{V}^{(j)}_{\varepsilon}(t, \cdot)\right\|_{H^s} = \left\|  \varphi_{\alpha_j, p}(   e^{-\im c_j t} e^{\im x}) \right\|_{H^s}  \sim 1, \ j=1,2,$$
as $\varepsilon \searrow 0$ (more detailed analysis of these kind of  Sobolev norms of can be found in Lemma \ref{vest}  below) and in  a similar way representing
$$ \mathcal{ V}^{(1)}_\varepsilon(0,x)- \mathcal{ V}^{(2)}_\varepsilon(0,x) =  \varphi_{\alpha_1 - \alpha_2, p}(  e^{\im x}), $$
so
$$ \Big\| \mathcal{V}^{(1)}_{\varepsilon}(0, \cdot)- \mathcal{V}^{(2)}_{\varepsilon}(0, \cdot)\Big\|_{H^s}  \sim \delta(\varepsilon) = |\log \varepsilon|^{-1/4} , \ j=1,2.$$
Further, we have \eqref{eq.I7} so for any $t >0$ we have
\begin{eqnarray}\label{eq.Pr1}
    \Big| \langle \mathcal{ V}^{(1)}_\varepsilon(t,\cdot), \mathcal{ V}^{(2)}_\varepsilon(t,\cdot) \rangle_{H^s}  \Big| =
    \Big| \sum_k  (1+|k|^2)^s \widehat{\mathcal{ V}^{(1)}_\varepsilon}(t,k) \overline{\widehat{\mathcal{ V}^{(2)}_\varepsilon}(t,k)} \Big| \sim
     \\ \nonumber \sim \varepsilon^{2s+1} \ \Big|\sum_k  (1+|k|^2)^s  e^{-\im (c_1 -c_2) t} (1-\varepsilon)^k \Big| \sim
     \Big( \frac{\varepsilon}{|c_1-c_2|t}\Big)^{1+2s}.
\end{eqnarray}
To this end we use the relation \eqref{eq.I3} and we find
$$ |c_1-c_2| \sim \varepsilon^{2s} |\log \varepsilon |^{-1/4} $$
so taking
$$ t=t_\varepsilon = \varepsilon^{1-2s} |\log \varepsilon |^{1/2},$$ via the property
$$ \frac{\varepsilon}{|c_1-c_2|t_\varepsilon } \sim |\log \varepsilon |^{-1/4} = \delta(\varepsilon) \to 0,$$
we obtain the relation
$$ \|\mathcal{V}^{(1)}_{\varepsilon}(t_\varepsilon, \cdot) -
   \mathcal{V}^{(2)}_{\varepsilon}(t_\varepsilon, \cdot)\|^2_{H^s} = \|\mathcal{V}^{(1)}_{\varepsilon}(t_\varepsilon, \cdot) \|^2_{H^s} +
   \|\mathcal{V}^{(2)}_{\varepsilon}(t_\varepsilon, \cdot) \|^2_{H^s} + o(1) \sim 2 + o(1)$$
   so we get \eqref{eq.con2}.  This completes the proof.
\end{proof}

The solutions $\mathcal{V}^{(i)}_{\varepsilon}(t, x)$ will play a crucial role along the
proof of Theorem \ref{theMi}.
The Szeg\H{o} equation has a remarkably deep structure, see \cite{GG2009,GG3,GG4,GG5,GG6,Po, Po2}. These aspects of the Szeg\H{o} equation are however not of importance for our analysis.

\section{A reduction of the problem}
First notice that if $\mathcal{V}(t, x)$ solves \eqref{eq.I0} then $v(t, x)=\mathcal{V}(t, x-t)$ solves
\begin{equation*}
\begin{cases}
    \im (\partial_t + \partial_x) v = P_{\geq 0} (v |v|^2),  \quad \quad (t, x)\in \R\times \R/(2\pi \Z)\\
    v(0,x)=f(x)\in H^s.
    \end{cases}
\end{equation*}
In particular we get
\begin{equation*}
\begin{cases}
    \im (\partial_t + \partial_x) v^{(j)}_\varepsilon = P_{\geq 0} (v^{(j)}_\varepsilon |v^{(j)}_\varepsilon|^2)\\
    v(0,x)={\mathcal V}^j_\varepsilon(0, x)\in H^s\,,
    \end{cases}
\end{equation*}
where
$v^{(j)}_\varepsilon (t, x)= \mathcal{V}^{(j)}_\varepsilon (t, x-t)$ and
$\mathcal{V}^{(j)}_\varepsilon(t, x)$, $j=1,2$ are the solutions constructed in the Section~\ref{sez}.
More precisely we have
$$v_\varepsilon^{(j)} (t,x) =  \frac{ \alpha_j e^{-\im t \omega_j}}{
1-p e^{\im (x-t(1+c_j))}}, \quad \quad j=1,2,
$$
where $ p(\varepsilon), \alpha_j(\varepsilon)$ are given by \eqref{eq.I4}, \eqref{eq.I5}
and $c_j(\varepsilon), \omega_j(\varepsilon),$ are obtained via \eqref{eq.I3}.
It is important to keep in mind that in the construction of ${\mathcal V}_\varepsilon^{(j)}(t, x)$
(and hence of $v_\varepsilon^{(j)}(t, x)$) we have
the following asymptotic of the parameters:
$$
p(\varepsilon)=\sqrt{1-\varepsilon}
\hbox{ and } \alpha_j(\varepsilon)=\varepsilon^{s+1/2} + o(\varepsilon^{s+1/2}),
\quad \quad  j=1,2.
$$
We can conclude the proof of Theorem \ref{theMi} provided that we can show that the solutions
$u^{(j)}_\varepsilon(t, x)$ to the Cauchy problems \eqref{eq:PotNLW}
such that:
$$u_\varepsilon^{(j)}(0, x)= {\mathcal V}^j_\varepsilon(0, x)$$
satisfy
$$
\lim_{\varepsilon\rightarrow 0} \ \ 
\sup_{t\in [0, t_\varepsilon]}\|u_\varepsilon^{(j)}(t, \cdot) - v_\varepsilon^{(j)}(t, \cdot)\|_{H^s}=0, \ \ t_\varepsilon = \varepsilon^{1-2s} \left| \log \left( \varepsilon \right) \right|^{1/2}.
$$

Hence Theorem \ref{theMi} follows from the following proposition.
\begin{proposition} \label{Tsm} Given any $ s \in (1/4,1/2)$ one can find  $\varepsilon_0 >0$ so that for any $\varepsilon \in (0,\varepsilon_0)$
there exists a solution $$ u_{\varepsilon}(t,x)     \in \mathcal{C}
([0,t_\varepsilon]; H^s), \ t_\varepsilon = \varepsilon^{1-2s} \left| \log \left( \varepsilon \right) \right|^{1/2} $$
to \eqref{eq:PotNLW}, satisfying the following conditions:
\begin{enumerate}[noitemsep,label=\alph*)]
\item
  $u_{\varepsilon}(0,x) = \frac{\alpha}{1-p e^{\im x}}$
  where $$ \alpha= \alpha_\varepsilon=\varepsilon^{s+1/2} + o\left( \varepsilon^{s+1/2}\right), \  p=
p_\varepsilon=\sqrt{1-\varepsilon};$$
\item
   $  \sup_{t \in [0, t_\varepsilon]}\|u_{\varepsilon}(t, \cdot) -v_{\varepsilon}(t, \cdot)\|_{H^s} \lesssim \varepsilon^{(4s-1)/4},$
where
\begin{equation}\label{eq.I22}
    v_{\varepsilon}(t,x) =   \frac{\alpha_\varepsilon e^{-\im \omega_\varepsilon t}}
    {1- p_\varepsilon e^{\im (x-t(1+c_\varepsilon))}}
\end{equation}
and $\omega_\varepsilon, c_\varepsilon$ are related to $\alpha_\varepsilon,
p_\varepsilon$ as in \eqref{eq.I3}.
\end{enumerate}
\end{proposition}

\section{Smoothing effect and a-priori estimates for $v_\varepsilon(t,x)$.}

Our aim in this section is to estimate the $L^\infty$ and $ H^\sigma$ norms
of $v_\varepsilon(t, x)$ as well as the $ H^\sigma$ norm of the action of the Duhamel operator associated with $e^{\im t|D_x|}$
on the expression
$$P_{<0}\left(v_\varepsilon |v_\varepsilon|^2\right),$$
where $P_{<0}$ is the projection in the negative frequencies and $v_\varepsilon(t, x)$
are given in Proposition \ref{Tsm}. The results of this section will be crucial along the proof
of Proposition \ref{Tsm} in section \ref{pro}.\\
\\
Our first step is to get in an explicit expression for
$P_{<0}\left(v_\varepsilon |v_\varepsilon|^2\right)$.

\begin{lemma} \label{l.G1} Let $\mathcal{A} \in \mathbb{C} $ and $\mathcal{P} \in \mathbb{C}$ satisfy
\begin{equation}\label{eq.SM1}
    |\mathcal{A}|=\alpha >0, \ |\mathcal{P}|=p \in (0,1),
\end{equation}
and
\begin{equation}\label{eq.S1am}
   v(x) \equiv \frac{\mathcal{A}}{1-\mathcal{P} e^{\im x}} = \mathcal{A} \left(\sum_{k=0}^\infty \mathcal{P}^k e^{\im kx}  \right).
\end{equation}
Then we have
\begin{equation}\label{S.3m}
   P_{<0} \left( v|v|^2\right) = \sum_{k=1}^\infty \frac{\mathcal{A}\alpha^2 }{(1-p^2)^2} \overline{\mathcal{P}}^k e^{-\im kx}.
\end{equation}
\end{lemma}
\begin{proof}

>From \eqref{eq.S1am}
we have
$$  v(x) |  v(x)|^2 = \Big( \frac{F(z)}{z - \overline{\mathcal{P}}}\Big)_{|z = e^{\im x}},$$
where $ F(z) = \frac{\mathcal{A}\alpha^2 z}{(1-\mathcal{P}z)^2}$
is analytic in a small neighborhood of the disc $\{|z| \leq 1\}$.\\
By the identity
$$\frac{F(z)}{z - \overline{\mathcal{P}}} = \frac{F(z)-F\left( \overline{\mathcal{P}} \right)}{z - \overline{\mathcal{P}}} + \frac{F\left( \overline{\mathcal{P}} \right)}{z - \overline{\mathcal{P}}},$$
and by noticing that the function
$H(z) = \frac{F(z)-F\left( \overline{\mathcal{P}} \right)}{z - \overline{\mathcal{P}}}$
is analytic in a small neighborhood of the disc $\{|z| \leq 1 \}$ (and in particular
$  \ P_{<0} \left(H(e^{\im x}) \right)=0$) we get
$$  P_{<0} \left( \frac{F(e^{\im x})}{e^{\im x} - \overline{\mathcal{P}}} \right) = F(\overline{\mathcal{P}})\  P_{<0} \left( \frac{1}{e^{\im x} - \overline{\mathcal{P}}} \right).$$
We conclude since we have
$$  F(\overline{\mathcal{P}}) = \frac{\mathcal{A}\alpha^2 \overline{\mathcal{P}}}{(1-|p|^2)^2}
$$ and $$ P_{<0} \left( \frac{1}{e^{\im x} - \overline{\mathcal{P}}} \right) = P_{<0} \left( \frac{{e^{-\im x}}}{1 - \overline{\mathcal{P}} {e^{-\im x}}} \right) = \sum_{k=0}^\infty \overline{\mathcal{P}}^{k} e^{-\im (k+1)x}.$$
\end{proof}
The next result will be crucial in the sequel.
\begin{lemma} \label{l.Iit} Let $\sigma\in [0, 1)$ and $ s \in (0,1/2)$
satisfy  one of the following conditions:
\begin{enumerate}[noitemsep,label=\alph*)]
\item $\sigma\in [0, 1/2), s\in (0, 1/2)$;
\item
$\sigma\in [1/2, 1/(4s)), s\in (1/4, 1/2).$
\end{enumerate}
Then there exists $\varepsilon_0 >0$
so that for any $\varepsilon \in (0,\varepsilon_0)$
we have:
\begin{equation*}
     \sup_{t\in [0,1]}\big \|\int_0^t e^{-\im (t-\tau) |D_x|}(P_{<0}\left(|v_\varepsilon(\tau,\cdot)|^2 v_\varepsilon(\tau,\cdot)\right) d\tau\big \|_{H^\sigma
     } \lesssim \varepsilon^{(3s-1/2)+\sigma(2s-1)}.
\end{equation*}
\end{lemma}

\begin{proof}
We have to estimate the $ H^\sigma$ norm of
\begin{equation}\label{eq.S10}
   w^{(0)}_\varepsilon(t,\cdot)  =  \int_0^t  \mathcal{U}(t-\tau) F_\varepsilon (\tau,\cdot) d\tau,
\end{equation}
where
$  \mathcal{U}(t) = e^{-\im t |D_x|} $
and
$ F_\varepsilon(t,\cdot) = P_{<0}\left(|v_\varepsilon(t,\cdot)|^2)
v_\varepsilon(t,\cdot)\right),$
or equivalently the quantity
$
\|\mathcal{W}^{(0)}_\varepsilon(t,\cdot)\|_{ H^\sigma},
$
where
\begin{equation}\label{eq.S12}
   \mathcal{W}^{(0)}_{\varepsilon} (t,\cdot) = \mathcal{U}(-t)w^{(0)}_\varepsilon
   (t,\cdot) = \int_0^t \mathcal{U}(-\tau) F_\varepsilon (\tau,\cdot) d\tau.
\end{equation}
We quote Lemma \ref{l.G1}, so taking
$$ \mathcal{A} = \mathcal{A}_\varepsilon=\alpha_\varepsilon e^{-\im \omega_\varepsilon t}
\hbox{ and } \mathcal{P}=\mathcal{P}_\varepsilon = p_\varepsilon e^{-\im t(1+c_\varepsilon)}, $$
we find
$$F_\varepsilon(t,x)=  \frac{\alpha_\varepsilon^3 e^{-\im \omega_\varepsilon t} }{(1-p_{\varepsilon}^2)^2} \left(\sum_{k=1}^\infty p_\varepsilon^k e^{\im k(1+c_\varepsilon)t} \ e^{-\im kx}\right).$$
The following identity is trivial:
$$\mathcal{U}(-\tau) \left( e^{-\im kx} \right) = e^{\im \tau |D_x|} \left( e^{-\im kx} \right) = e^{\im \tau |-k|}  e^{-\im kx}= e^{\im \tau k} e^{-\im kx},
\quad \quad \forall k\geq 0,$$
and hence
$$ \mathcal{W}^{(0)}_{\varepsilon}(t,x) = \sum_{k=1}^\infty \widehat{\mathcal{W}}_{\varepsilon}(t,k) e^{-\im kx},  $$
where
\begin{equation}\label{eq.S14}
  \widehat{\mathcal{W}}_{\varepsilon}(t,k) =  \frac{\alpha_{\varepsilon}^3 }{(1-p_{\varepsilon}^2)^2}  p_{\varepsilon}^k  \Lambda(t,k, \omega_{\varepsilon},c_{\varepsilon})
\end{equation}
and
$$ \Lambda(t,k, \omega_{\varepsilon},c_{\varepsilon})= \int_0^t  e^{-\im \tau (\omega_{\varepsilon} -k(2+c_{\varepsilon}))}  d\tau = - \frac{e^{-\im t(\omega_{\varepsilon}-k(2+c_{\varepsilon}))}-1}{\im (\omega_{\varepsilon}-k(2+c_{\varepsilon})) }.$$
We shall use  the estimate
\begin{equation}\label{eq.S15}
   |\Lambda(t,k, \omega_{\varepsilon},c_{\varepsilon})| \lesssim
   \frac{1}{1+|\omega_{\varepsilon}-k(2+c_{\varepsilon})|}, \ \forall t \in [0,1],
\end{equation}
together with the identity $1-p_{\varepsilon}^2=\varepsilon$ and the asymptotic expansions
\begin{equation}\label{eq.S18}
\alpha_{\varepsilon} = \varepsilon^{s+1/2} + o\left( \varepsilon^{s+1/2}\right),\quad
\omega_{\varepsilon} = \varepsilon^{2s-1} +o \left( \varepsilon^{2s-1}\right),\quad  c_{\varepsilon}= \varepsilon^{2s} +o \left( \varepsilon^{2s}\right).
\end{equation}
Thus we can take $t \in  [0,1],$ $ \sigma\geq 0$ and we can derive the relations
\begin{equation}\label{eq.S16}
 \|\mathcal{W}_{\varepsilon}^{(0)}(t,\cdot)\|^2_{H^\sigma} = \sum_{k=1}^\infty \left| \widehat{\mathcal{W}}_{\varepsilon}(t,k) \right|^2 k^{2\sigma}
 \lesssim \varepsilon^{6s-1}\sum_{k=1}^\infty   \frac{k^{2\sigma}(1-\varepsilon)^k}{1+|\omega_{\varepsilon}-k(2+c_{\varepsilon})|^2}\,.
\end{equation}
Further, we can observe that
 $$ \frac{1}{1+|\omega_{\varepsilon}-k(2+c_{\varepsilon})|^2}  \lesssim \frac{1}{2+|\tilde {\omega}_{\varepsilon}-k|^2}, \ \ \tilde {\omega}_{\varepsilon}= \frac{\omega_{\varepsilon}}{2+c_{\varepsilon}} $$
 and therefore
 $$   \frac{1}{1+|\omega_{\varepsilon}-k(2+c_{\varepsilon}
 )|^2} \lesssim    \frac{1}{1+|N_{\varepsilon}-k|^2}  ,$$
 where
 \begin{equation}\label{eq.S18ll}
   N_\varepsilon= \left[\frac{\omega_{\varepsilon}}{2+c_{\varepsilon}}\right] \sim
   \varepsilon^{2s-1} + o\left(\varepsilon^{2s-1} \right)
 \end{equation}
and $[a]$ is the integer part of the real number $a$, i.e. $[a] \leq a < [a]+1.$
Next notice that for any integer $N \geq 1$ we have
\begin{equation}\label{eq.SS18a2}
   \sum_{k=1}^{10N}   \frac{k^{2\sigma}(1-\varepsilon)^k}{1+|N-k|^2} \lesssim
   \sum_{k=1}^{10N}  \frac{N^{2\sigma}}{1+|N-k|^2}\lesssim N^{2\sigma}\,.
\end{equation}
On the other hand for  $0 \leq 2\sigma  < 1$ we have
\begin{equation}\label{gigl}\sum_{k=10N}^\infty   \frac{k^{2\sigma}}{1+|N-k|^2} \lesssim
\sum_{k=10N}^\infty   \frac{1}{k^{2-2\sigma}} \lesssim  N^{2\sigma}.\end{equation}
We conclude the proof in the case $a)$ by combining \eqref{eq.S16},
\eqref{eq.S18ll},
\eqref{eq.SS18a2} and \eqref{gigl}.\\ \\
In the case $1/2 \leq \sigma < 1$ we can use the H\"older inequality
$$ \sum_{k=10N}^\infty   \frac{k^{2\sigma}(1-\varepsilon)^k}{1+|N-k|^2} \leq  \left(\sum_{k=10N}^\infty k^{2(\sigma-1)p} \right)^{1/p} \left(\sum_{k=10N}^\infty (1-\varepsilon)^{kq} \right)^{1/q},$$
where
\begin{equation}\label{hold}1 < p < q < \infty, \ \ \frac{1}{p} + \frac{1}{q} =1.
\end{equation}
The convergence of the series
$$ \sum_{k=10N}^\infty k^{2(\sigma-1)p}  < \infty$$
is fulfilled if
\begin{equation}\label{eq.SS18a1}
  2(\sigma-1)p < -1 \ \Longleftrightarrow \ \ \frac{1}{p} < 2(1-\sigma).
\end{equation}
The inequality
$$ \sum_{k=10N}^\infty (1-\varepsilon)^{kq}  \leq  \sum_{k=0}^\infty (1-\varepsilon)^{kq}\leq \frac{1}{1-(1-\varepsilon)^q} \lesssim \frac{1}{\varepsilon}$$
implies
$$ \sum_{k=10N}^\infty   \frac{k^{2\sigma}(1-\varepsilon)^k}{1+|N-k|^2} \lesssim \frac{1}{\varepsilon^{1/q}}.$$
Summarizing, we get
\begin{equation}\label{staf} \sum_{k=10N}^\infty   \frac{k^{2\sigma}(1-\varepsilon)^k}{1+|N-k|^2} \lesssim N^{2\sigma}
\end{equation}
provided that $1/2\leq \sigma<1$ is such that
there exist $1/p$ and $1/q$ that satisfy \eqref{hold} and moreover
\begin{equation}\label{eq.SS18s5}
   \frac{1}{p} < 2(1-\sigma), \ \frac{1}{q} < (1-2s)2\sigma.
\end{equation}
The conditions are satisfied if $1/2\leq \sigma < 1/(4s)$.
We conclude the proof in the case $b)$ by combining
\eqref{eq.S16},
\eqref{eq.S18ll},
\eqref{eq.SS18a2} and \eqref{staf}.

\end{proof}

We conclude this section with the estimate of the
$ H^\sigma$ and $L^\infty$ norms of $v_{\varepsilon}(t, x)$.

\begin{lemma} \label{vest} For any
$s \in \left(0, \frac{1}{2} \right), \sigma \in [0, 1]$
we have the estimates:
\begin{equation}\label{eq.I51in}
  \sup_{t\in [0,1]}   \|v_\varepsilon(t, \cdot)\|_{L^\infty} \lesssim \varepsilon^{s-1/2}
\end{equation}
and
\begin{equation}\label{eq.I51}
     \sup_{t\in [0,1]}\|v_\varepsilon(t, \cdot)\|_{H^\sigma} \lesssim \varepsilon^{s-\sigma}.
\end{equation}
\end{lemma}
 \begin{proof}
We have
 $$ v_\varepsilon(t,x) = \sum_{k=0}^\infty \widehat{v}_{\varepsilon}(t,k) e^{\im kx}, $$
 where
 $$
   \left| \widehat{v}_{\varepsilon}(t,k)\right| \lesssim \varepsilon^{s+1/2} (1-\varepsilon)^{k/2} , \  \ \forall t \in [0,1].
$$
The estimate \eqref{eq.I51in} follows by the Minkowski inequality
and the following estimate $$\sum_{k\geq 0} (1-\varepsilon)^{k/2}\lesssim \varepsilon^{-1}.$$
To prove \eqref{eq.I51} we have to show the estimate
\begin{equation}\label{equival}
\sum_{k\geq 0} (1-\varepsilon)^{k} k^{2\sigma}\lesssim \varepsilon^{-2\sigma-1}, \quad \quad \forall \sigma\in [0, 1].
\end{equation}
Estimate \eqref{equival} for $\sigma=0$ is straightforward. Next, we notice that
\begin{align*}&\sum_{k\geq 0} (1-\varepsilon)^{k} k^{2\sigma}\lesssim
\frac{N^{2\sigma}}{\varepsilon} + (1-\varepsilon)^N \sum_{h=0}^\infty (1-\varepsilon)^h (N+h)^{2\sigma}, \quad \quad \forall N.
\end{align*}
In particular
for $\sigma=1$ we write, for $N\geq 1$ to be chosen:
\begin{multline*}
\sum_{k\geq 0} (1-\varepsilon)^{k} k^{2}\lesssim
\frac{N^{2}}{\varepsilon} + (1-\varepsilon)^N \sum_{h=0}^\infty (1-\varepsilon)^h (N+h)^{2}
\\
= \frac{N^{2}}{\varepsilon} + (1-\varepsilon)^N \sum_{h=0}^\infty (1-\varepsilon)^h N^2 + (1-\varepsilon)^N \sum_{h=0}^\infty (1-\varepsilon)^h h^{2}
 + 2(1-\varepsilon)^N \sum_{h=0}^\infty (1-\varepsilon)^h hN
\end{multline*}
and hence by elementary computations:
$$
\sum_{k\geq 0} (1-\varepsilon)^{k} k^{2}\lesssim
\frac{N^{2}}{\varepsilon} +\frac 1{\varepsilon^3}+\frac N{\varepsilon^2}.
 $$
By choosing $N=\big[
\frac 1{\varepsilon}\big]$
we deduce \eqref{equival} for $\sigma=1$.
The proof of \eqref{equival} for $\sigma\in (0, 1)$ follows by interpolation between
$\sigma=0$ and $\sigma=1$.

 \end{proof}
\section{Proof of Proposition \ref{Tsm}}\label{pro}
Given any $\varepsilon \in (0,1)$ we define the solution to the cubic Szeg\"o equation $v=v_\varepsilon$ by
\begin{equation*}
     v_{\varepsilon}(t,x) =   \frac{\alpha e^{-\im \omega t}}{1- p e^{\im (x-t(1+c))}},
\end{equation*}
where the parameters $\alpha,p$ are chosen as follows
$$ \alpha= \varepsilon^{s+1/2} + o\left( \varepsilon^{s+1/2}\right), \  p=\sqrt{1-\varepsilon}$$
and $\omega, c$ are determined as in \eqref{eq.I3}.
Then we look for solutions to \eqref{eq:PotNLW}
as a perturbation of $ v_{\varepsilon}(t,x)$:
\begin{equation*}
u_\varepsilon(t,x) =  v_{\varepsilon}(t,x) +  w_{\varepsilon}(t,x),
\end{equation*}
so that $w(t,x) = w_{\varepsilon}(t,x)$ has to be  a solution to the equation
\begin{equation}\label{eq:PotNLWm21}
 (\im  \partial _t  - |D _x|    ) w =
  \left(|v+  w| ^{2} (v+ w)  - P_{\geq 0}(|v|^2) v\right)
\end{equation}
with zero initial data.
Turning back to \eqref{eq:PotNLWm21}, we
can rewrite it as follows
\begin{equation}\label{eq.s1modw}
     (\im  \partial _t  -     |D_x|) w  =   (w+v_\varepsilon)^2\overline{(w+v_\varepsilon)}- v_\varepsilon^2\overline{v_\varepsilon} + P_{<0}(v_\varepsilon^2 \overline{v_\varepsilon}).
\end{equation}
It is important to classify all term in the right sides of \eqref{eq.s1modw}. We have linear combinations of the following terms:
\begin{enumerate}[noitemsep,label=\alph*)]
\item  Term $ w^2 \overline{w}$ cubic in $w.$
  \item Terms $ w^2 \overline{v_\varepsilon}$ and $ w \overline{w} v_\varepsilon$  quadratic in $w.$
  \item Terms $  w v_\varepsilon \overline{v_\varepsilon}$ and  $\overline{w} v_\varepsilon^2$ linear in $w.$
  \item Term  of type $ P_{<0}(v_\varepsilon v_\varepsilon  \overline{ v_\varepsilon})= P_{<0}(v_\varepsilon |v_\varepsilon|^2)$.
\end{enumerate}

\begin{lemma} \label{l.Q1}
For every $\sigma>1/2$, $s\in (1/4, 1/2)$ and for every $t\in [0,1]$ we have
\begin{multline}\label{eq.qe1ab}
 \|w(t) v_\varepsilon(t) \overline{v_\varepsilon}(t)\|_{H^\sigma} +
\|\overline{w}(t) (v_\varepsilon(t))^2 \|_{H^\sigma}
\\
\lesssim \varepsilon^{2(s-1/2)} \|w(t)
\|_{H^\sigma}+ \varepsilon^{2s-\sigma-1/2}
\|w(t)\|^{1-1/(2\sigma)}_{L^2} \|w(t)\|^{1/(2\sigma)}_{H^{\sigma}}
\end{multline}
and also
\begin{multline}\label{eq.qe1abs}
\|w(t) v_\varepsilon(t) \overline{v_\varepsilon}(t)\|_{L^2}
+\|\overline{w}(t) (v_\varepsilon(t))^2 \|_{L^2}
\\
\lesssim
 \varepsilon^{2s-1} \|w(t)\|_{L^2}  +  \varepsilon^{2s-1/2} \|w(t)\|^{1-1/(2\sigma)}_{L^2}
 \|w(t)\|_{H^{\sigma}}^{1/(2\sigma)}\,.
\end{multline}
\end{lemma}
\begin{proof}
By the Minkowski inequality and the well-known estimate
\begin{equation}\label{algebra}
\|f g\|_{H^\sigma}\lesssim \|f\|_{ H^\sigma} \|g\|_{L^\infty}
+  \|g\|_{ H^\sigma} \|f\|_{L^\infty}
\end{equation}
we get
$$
 \|w(t) v_\varepsilon(t) \overline{v_\varepsilon}(t)\|_{H^\sigma}
 +\|\overline{w}(t) (v_\varepsilon(t))^2 \|_{H^\sigma}
 \lesssim \underbrace{\|v_\varepsilon\|_{L^\infty}^2 \|w(t)\|_{H^\sigma}}_{I} + \underbrace{\|v_\varepsilon\|_{L^\infty} \|v_\varepsilon\|_{H^\sigma} \|w(t)\|_{L^\infty}}_{II}.
 $$
For the first term we use the estimate \eqref{eq.I51in} and deduce
$$ I \lesssim \varepsilon^{2(s-1/2)} \|w(t)
\|_{H^\sigma}.$$
To estimate $II$ we need the interpolation inequality
\begin{equation}\label{eq.SC1}
    \|f\|_{L^\infty} \lesssim \|f\|^{1-1/(2\sigma)}_{L^2} \|f\|^{1/(2\sigma)}_{H^{\sigma}},
\end{equation}
where $\sigma>1/2.$
We obtain
$$ II \lesssim \varepsilon^{s-1/2} \varepsilon^{s-\sigma}
\|w(t)\|^{1-1/(2\sigma)}_{L^2} \|w(t)\|^{1/(2\sigma)}_{H^{\sigma}},$$
where we used \eqref{eq.I51in} and \eqref{eq.I51}.
We also have
$$
\|w(t) v_\varepsilon(t) \overline{v_\varepsilon}(t)\|_{L^2} +
\|\overline{w}(t) (v_\varepsilon(t))^2
\|_{L^2} \lesssim
\underbrace{\|v_\varepsilon\|_{L^\infty}^2 \|w(t)\|_{L^2}}_{I'}
+ \underbrace{ \|v_\varepsilon\|_{L^\infty} \|v_\varepsilon\|_{L^2} \|w(t)\|_{L^\infty}}_{II'}.$$
Arguing as above we get
$$I' \lesssim \varepsilon^{2s-1}  \|w(t)\|_{L^2}
$$
and
$$II' \lesssim \varepsilon^{2s-1/2} \|w(t)\|_{L^\infty}
\lesssim \varepsilon^{2s-1} \|w(t)\|^{1-1/(2\sigma)}_{L^2}
\|w(t)\|_{H^{\sigma}}^{1/(2\sigma)}.$$

\end{proof}

\begin{lemma} \label{l.Q2}
For any $\sigma>1/2$, $s\in (1/4, 1/2)$ and for every $t \in [0,1]$ we have
\begin{multline}\label{eq.q21}
\|(w(t))^2 \overline{v_\varepsilon}(t)
\|_{H^\sigma} + \|w(t) \overline{w}(t) v_\varepsilon(t)\|_{H^\sigma} \\
\lesssim
\varepsilon^{(s-1/2)} \|w(t)
\|_{H^\sigma}^{1+1/(2\sigma)} \|w(t)\|_{L^2}^{1-1/(2\sigma)}+
\varepsilon^{s-\sigma}
\|w(t)\|_{H^{\sigma}}^{1/\sigma} \|w(t)\|_{L^2}^{2-1/\sigma}
\end{multline}
and also
\begin{multline}\label{eq.q22}
 \|(w(t))^2 \overline{v_\varepsilon}(t)\|_{L^2} + \|2 w(t) \overline{w}(t) v_\varepsilon(t)\|_{L^2}
\\
\lesssim   \varepsilon^{s-1/2} \|w(t)\|_{L^2}^{2- 1/(2\sigma)}
 \|w(t)\|_{H^\sigma }^{1/(2\sigma)}
  +  \varepsilon^{s} \|w(t)\|^{2-1/\sigma}_{L^2}
\|w(t)\|_{H^{\sigma}}^{1/\sigma}.
\end{multline}
\end{lemma}

\begin{proof}
By the Minkowski inequality and \eqref{algebra} we get
$$ \|(w(t))^2 \overline{v_\varepsilon}(t)\|_{ H^\sigma} +
\|w(t)\overline{w}(t) v_\varepsilon(t) \|_{H^\sigma}  \lesssim \underbrace{\|v_\varepsilon(t)\|_{L^\infty}
\|w(t)\|_{L^\infty} \|w(t)\|_{H^\sigma}}_{I} +
\underbrace{\|v_\varepsilon(t)\|_{ H^\sigma}
\|w(t)\|_{L^\infty}^2}_{II}
$$
For the first term we use the estimate \eqref{eq.I51in} together with the estimate
\eqref{eq.SC1} and deduce
$$ I \lesssim \varepsilon^{(s-1/2)} \|w(t)
\|_{H^\sigma}^{1+1/(2\sigma)} \|w(t)\|_{L^2}^{1-1/(2\sigma)}.$$
We also get
$$ II \lesssim \varepsilon^{s-\sigma}
\|w(t)\|_{H^{\sigma}}^{1/\sigma} \|w(t)\|_{L^2}^{2-1/\sigma} ,$$
where we used \eqref{eq.I51} and we conclude
the first estimate of the lemma.
We also have
$$\|(w(t))^2 \overline{v_\varepsilon}(t)\|_{L^2} + \|w(t)\overline{w}(t) v_\varepsilon(t)
\|_{L^2} \lesssim \underbrace{\|v_\varepsilon(t)\|_{L^\infty}
\|w(t)\|_{L^\infty} \|w(t)\|_{L^2}}_{I'}
+ \underbrace{\|v_\varepsilon(t)\|_{L^2} \|w(t)\|_{L^\infty}^2}_{II'},
$$
and arguing as above we get
$$I' \lesssim \varepsilon^{s-1/2}   \|w(t)\|_{ H^\sigma }^{1/(2\sigma)}
\|w(t)\|_{L^2}^{2-1/(2\sigma)},\quad
II' \lesssim \varepsilon^{s} \|w(t)\|^{2-1/\sigma}_{L^2}
\|w(t)\|^{1/\sigma}_{H^{\sigma}}.
$$

\end{proof}

\begin{lemma} \label{l.Q3}
For any $\sigma>1/2$ and for every $t \in [0,1]$ we have
\begin{equation}\label{eq.q31}
\|(w(t))^2 \bar w(t)\|_{H^\sigma} \lesssim
\|w(t)\|_{L^2}^{2-1/\sigma}\|w(t)\|_{H^\sigma}^{1+1/\sigma}
\end{equation}
and also
\begin{equation}\label{eq.q32}
\|(w(t))^2 \bar w(t)\|_{L^2} \lesssim
\|w(t)\|_{L^2}^{3-1/\sigma}  \|w(t)\|_{H^\sigma}^{1/\sigma}.
\end{equation}
\end{lemma}

\begin{proof}
By \eqref{algebra} we get
$$ 
\|(w(t))^2 \overline {w}(t) \|_{H^\sigma}
\lesssim \|w(t)\|_{L^\infty}^2
\|w(t)\|_{H^\sigma}\lesssim
\|w(t)\|_{L^2}^{2-1/\sigma}\|w(t)\|_{H^\sigma}^{1+1/\sigma}, 
$$
where we used
\eqref{eq.SC1} at the last step.
We also have
$$ \|(w(t))^2 \overline {w}(t) \|_{L^2}
\lesssim \|w(t)\|_{L^\infty}^2
\|w(t)\|_{L^2}\lesssim
\|w(t)\|_{L^2}^{3-1/\sigma} \|w(t)\|_{H^\sigma}^{1/\sigma}.$$
\end{proof}
In order to conclude the proof of Proposition~\ref{Tsm}, we need the following Gronwall type lemma.
\begin{lemma}
\label{L.gr1}
Let $ F : [0,\infty) \to [0,\infty)$ be a continuous function satisfying
\begin{equation}\label{eq.Gr1}
    F(u) \leq C u , \ \ \forall  u \in [0,1] 
\end{equation}
and let
$g_\varepsilon (t)$ be a family of continuous and non-negative functions satisfying
\begin{equation}\label{eq.SS19}
  g_\varepsilon (t) \lesssim\varepsilon^{\theta} +  \int_0^t F(g_\varepsilon(s)) \frac{ds}{\varepsilon} ,
\end{equation}
with $\theta>0$.
Then there exists $\varepsilon_0>0$ so that we have the inequality
\begin{equation}\label{eq.gr0}
    g_\varepsilon (t) \lesssim \varepsilon^{\theta/2} \quad \quad
\forall t \in [0,  \varepsilon |\log \varepsilon |^{1/2}], \quad \quad \forall \varepsilon\in (0, \varepsilon_0).
\end{equation}
\end{lemma}

\begin{remark}A similar analysis has been done in \cite{BGT3, Tz04}, where one can find a version of Lemma~\ref{L.gr1} with 
the assumption 
\begin{equation}\label{eq.SS19m}
  g^\prime_\varepsilon (t) \lesssim \varepsilon^{\theta} +   \frac{F(g_\varepsilon(t))}\varepsilon,
\end{equation}
in the place of \eqref{eq.SS19}. In this case the assumption $\theta >0$ in Lemma \ref{L.gr1} is transformed into $\theta>-1.$
It is not important that the first term in the right hand-side of \eqref{eq.SS19m} is small, the important point is that it is smaller than the amplification factor $\varepsilon^{-1}$
coming from the other terms in the right hand-side of \eqref{eq.SS19} (see \cite{BGT3, Tz04} for more details).
\end{remark}

\begin{proof}[Proof of Lemma~\ref{L.gr1}] We can introduce the rescaled functions
$y_\varepsilon(t)= \varepsilon^{-\theta/2} g_\varepsilon(\varepsilon t)$,
and we can deduce that  the assumption \eqref{eq.SS19} is equivalent to
\begin{equation}\label{eq.yi}
   y_\varepsilon (t) \leq C\big (\varepsilon^{\theta/2} +  \varepsilon^{-\theta/2}  \int_0^t F(\varepsilon^{\theta/2}y_\varepsilon(s)) ds \big), \ \ \ \forall t \in [0,|\log \varepsilon |^{1/2}].
\end{equation}
It is not difficult to see that $y_\varepsilon(0) \leq C\varepsilon^{\theta/2}  < 1$ for $\varepsilon $ small enough and hence  
$y_\varepsilon(t)  < 1$ for $t$ close to $0.$ The inequality \eqref{eq.gr0} will be established, if we can show that the graph of the function $y_\varepsilon(t)$ does not intersect the line $y=1$ for $t \in [0,|\log \varepsilon |^{1/2}].$ Hence it remains to show that
\begin{equation}\label{eq.Sol1}
   y_\varepsilon(t) < 1  \quad \quad \forall t\in [0,
 |\log \varepsilon |^{1/2}].
\end{equation}
Indeed, if this is not true let $t_1 \in (0, |\log \varepsilon |^{1/2})$ be the first point, where the graph of the function $y_\varepsilon(t)$ intersects the line $y=1,$
i.e.
\begin{equation}\label{eq.yi1}
   y_\varepsilon(t_1) = 1, \ y_\varepsilon(t) < 1, \ \ \forall t \in [0, t_1).
\end{equation}

Then $\varepsilon^{\theta/2}y_\varepsilon(s)< 1 $ for $s \in [0,t_1]$ so \eqref{eq.Gr1} and \eqref{eq.yi} imply 
$$ y_\varepsilon (t) \leq C \varepsilon^{\theta/2} +  C^2 \varepsilon^{-\theta/2}  \int_0^{t} \varepsilon^{\theta/2}y_\varepsilon(s) ds \leq 
C \varepsilon^{\theta/2} + C^2 \int_0^{t} y_\varepsilon(s) ds \ \ \forall t \in [0,t_1]. $$
Applying the classical Gronwall inequality in $[0,t_1],$ we get
$$ y_\varepsilon (t) \leq C \varepsilon^{\theta/2} e^{C^2t} \ \ \forall t \in [0,t_1]$$
and hence
$$ \log y_\varepsilon(t_1) \leq \log C + \frac{\theta \log \varepsilon}{2} + C^2 |\log \varepsilon |^{1/2} < 0$$
as $\varepsilon$ is sufficiently small. This is in  contradiction  with \eqref{eq.yi1}. The contradiction shows that
\eqref{eq.Sol1} holds  and this completes the proof.

\end{proof}
We can now conclude the proof of Proposition~\ref{Tsm}.
We introduce the functions
$$h_\varepsilon(t)= \varepsilon^{-s}\|w(t)\|_{L^2} + \varepsilon^{\sigma-s} \|w(t)\|_{ H^\sigma}$$
where $\sigma>1/2$ is any number that will be chosen properly at the end of the proof.
Then, thanks to the analysis performed in the beginning of this section and the previous one,
$$h_\varepsilon(t)\lesssim \varepsilon^{(2s-1/2)}
+
\varepsilon^{ 2s-1} \int_0^t h_\varepsilon (s) ds
+ \varepsilon^{2s-1} \int_0^t (h_\varepsilon (s))^2 ds+ \varepsilon^{2s-1} \int_0^t (
h_\varepsilon (s))^3 ds.
$$
By Lemma \ref{L.gr1} we get
$$h_\varepsilon (t)\lesssim \varepsilon^{s-1/4}, \quad \quad \forall t\in (0,
\varepsilon^{1-2s} |\log \varepsilon|^{1/2})$$
and hence
$$\|w(t, \cdot)\|_{L^2} \lesssim \varepsilon^{2s-1/4}, \|w(t, \cdot)\|_{ H^\sigma} \lesssim \varepsilon^{2s-\sigma-1/4}, \quad \quad \forall t\in (0, 
 \varepsilon^{1-2s} |\log \varepsilon|^{1/2}).$$
We finally can write
$$
\|w(t,\cdot)\|_{H^s}\leq \|w(t,\cdot)\|_{L^2}^{1-\frac{s}{\sigma}}\|w(t,\cdot)\|_{H^\sigma}^{\frac{s}{\sigma}}\lesssim \varepsilon^{s-1/4}\,
$$
and we conclude the proof.

\end{document}